\documentclass[final,leqno]{siamltex704}
\usepackage{amsmath}
\usepackage{graphicx}
\usepackage[notcite,notref]{showkeys}
\usepackage{mathrsfs}
\usepackage{graphics}
\usepackage{float}
\usepackage{amsfonts,amssymb}
\usepackage{dsfont}
\usepackage{pifont}
\usepackage{wrapfig} 
\usepackage{hyperref}
\usepackage{multirow}
\usepackage{color}
\numberwithin{equation}{section}
\usepackage{threeparttable}
\def\3bar{{|\hspace{-.02in}|\hspace{-.02in}|}}

\def\T{{\mathcal{T}}}

\def\pT{{\partial T}}

\def\bw{{\mathbf{w}}}

\def\bn{{\mathbf{n}}}

\def\bw{{\mathbf{w}}}

\def\bf{{\mathbf{f}}}

\def\ljump{{[\![}}
\def\rjump{{]\!]}}

\newtheorem{algorithm}{Weak Galerkin Algorithm}[section]

\title {A parallel iterative procedure for weak Galerkin methods for second order elliptic problems}

\begin{document}

 \author{
 Chunmei Wang \thanks{Department of Mathematics, University of Florida, Gainesville, FL 32611, USA (chunmei.wang@ufl.edu). The research of Chunmei Wang was partially supported by National Science Foundation Award DMS-2136380.}
 \and
 Junping Wang\thanks{Division of Mathematical
 Sciences, National Science Foundation, Alexandria, VA 22314
 (jwang@nsf.gov). The research of Junping Wang was supported in part by the
 NSF IR/D program, while working at National Science Foundation.
 However, any opinion, finding, and conclusions or recommendations
 expressed in this material are those of the author and do not
 necessarily reflect the views of the National Science Foundation.}
  \and
 Shangyou Zhang\thanks{Department of Mathematical Sciences,  University of Delaware, Newark, DE 19716, USA (szhang@udel.edu).  }
 }

\maketitle

\baselineskip=12pt

\begin{abstract}
A parallelizable iterative procedure based on domain decomposition is presented and analyzed for weak Galerkin finite element methods for second order elliptic equations. The convergence analysis is established for the decomposition of the domain into individual elements associated to the weak Galerkin methods or into larger subdomains. A series of numerical tests are illustrated to verify the theory developed in this paper.
\end{abstract}



\pagestyle{myheadings}

\section{Introduction}
This paper is concerned with an iterative procedure related to domain decomposition techniques based on the use of subdomains as small as individual elements for weak Galerkin (WG) methods for second order elliptic equations in $\mathbb R^d (d=2, 3)$. For simplicity, we consider the second order elliptic problem with Dirichlet boundary condition
\begin{equation}\label{model}
\begin{split}
-\nabla\cdot(a\nabla u)+cu=&f, \qquad\text{in} \ \Omega\subset\mathbb R^d,\\
u=&g, \qquad\text{on} \ \partial\Omega,
\end{split}
\end{equation}
where $d=2, 3$. Assume the coefficients $a(x)$ and $c(x)$ satisfy
$$
0<a_0\leq a(x) \leq a_1<\infty, \quad
0\leq c(x)\leq c_1<\infty,
$$
and are sufficiently regular so that the existence and uniqueness of a solution of \eqref{model}  in $H^s(\Omega)$ hold true for some $s>1$ for  reasonable $f$ and $g$. A weak formulation for the model problem \eqref{model} reads as follows: Find $u\in  H^1(\Omega)$ such that $u=g$ on $\partial\Omega$, satisfying
\begin{equation}\label{weakform}
(a\nabla u, \nabla v)+(cu, v)= (f, v)  \qquad \forall v\in H_0^1(\Omega).
\end{equation}

The WG finite element method is emerging as a new and efficient numerical method for solving PDEs. The idea of WG
method was first proposed by Wang and Ye for solving second
order elliptic equations in 2011 \cite{wysecond}. This method was
subsequently developed for various PDEs, see \cite{mwy1, mwy2, mwy3655,
mwy2015, mwyz-biharmonic, ww4, ww5, ww2016, ww6,
ww7,ww4, wysecond, wy3655, mwy3, wystokes, ww2, ww5, r23, rham}.
To our best knowledge, there have not been any iterative algorithms designed for the WG methods along the line of domain decompositions. Due to the large size of the computational problem, it is necessary and crucial to design efficient and parallelizable iterative algorithms for the WG scheme. Our iterative procedure is motivated by Despres \cite{9}  for a Helmholz problem and  another Helmholz-like problem related to Maxwell's equations by Despres \cite{10, 11}.  It should be noted that the convergence in \cite{9, 10, 11} were established for the differential problems in strong form where numerical results were presented to validate the iterative procedures for the discrete case. Douglas et al. \cite{j} introduced a parallel iterative procedure for the second order partial differential equations by using the mixed finite element methods. The goal of this paper is to extend the result of Douglas into weak Galerkin finite element methods. In particular, based on the features of weak Galerkin methods, the iterative procedure developed in this paper can be very naturally and easily implemented on a massively parallel computer by assigning each subdomain to its own processor. In addition, the work is based on the hybridized weak Galerkin framework of \cite{wang} extended to non-overlapping subdomains.

The paper is organized as follows. In Section \ref{Section:Hessian}, we briefly review the weak differential operators and their discrete analogies. In Section \ref{Section:WGFEM}, we describe the WG method for the model problem \eqref{model}. In Section \ref{Section:DD-WGFEM}, we introduce domain decompositions and derive a hybridized formulation for the WG method. In Section \ref{Section:IterativeScheme}, we present a parallel iterative procedure for the WG finite element method. In Section \ref{Section:Convergence}, we establish a convergence analysis for the parallel iterative scheme. Finally in Section \ref{Section:numerical-experiments}, we report several  numerical results to verify our convergence theory.

\section{Weak Differential Operators}\label{Section:Hessian}
The primary differential operator in the weak formulation (\ref{weakform}) for the second order elliptic problem (\ref{model}) is the gradient operator $\nabla$, for which a discrete weak version has been introduced in \cite{wy3655}. For completeness, let us briefly review the definition as follows.

Let $T$ be a polygonal or polyhedral domain with boundary $\partial T$. A weak function on $T$ refers to   $v=\{v_0,v_b\}$ where $v_0\in L^2(T)$ and $v_b\in L^{2}(\partial T)$  represent the values of $v$ in the interior and on the boundary of $T$ respectively. Note that $v_b$ may not necessarily be the trace of $v_0$ on $\partial T$. Denote by $W(T)$ the local space of weak functions on $T$; i.e.,
\begin{equation*}\label{2.1}
W(T)=\{v=\{v_0,v_b\}: v_0\in L^2(T), v_b\in
L^{2}(\partial T)\}.
\end{equation*}

The weak gradient of $v\in W(T)$, denoted by $\nabla_w v$, is defined as a linear functional on $[H^1(T)]^{d}$ such that
\begin{equation*}
(\nabla_w  v, \bw)_T=-(v_0,\nabla \cdot \bw)_T+\langle v_b,\bw\cdot  \textbf{n}\rangle_{\partial T} \qquad \forall \bw\in [H^1(T)]^{d}.
\end{equation*}

Denote by $P_r(T)$ the space of all polynomials on $T$ with total degree $r$ and/or less. A discrete version of $\nabla_{w}v$  for $v\in W(T)$, denoted by $\nabla_{w, r, T}v$, is defined as a unique polynomial vector in $[P_r(T) ]^{d}$ satisfying
\begin{equation}\label{disgradient}
(\nabla_{w, r, T}v, \bw)_T=-(v_0,\nabla \cdot \bw)_T+\langle v_b,\bw\cdot\textbf{n}\rangle_{\partial T}, \quad\forall\bw\in [P_r(T)]^{d}.
\end{equation}

\section{Weak Galerkin Algorithm}\label{Section:WGFEM}

Let ${\cal T}_h$ be a finite element partition of the domain $\Omega$ consisting of polygons or polyhedra that are shape-regular \cite{wy3655}. Denote by ${\mathcal E}_h$ the
set of all edges or flat faces in ${\cal T}_h$ and  ${\mathcal
E}_h^0={\mathcal E}_h \setminus
\partial\Omega$ the set of all interior edges or flat faces.
Denote by $h_T$ the meshsize of $T\in {\cal T}_h$ and
$h=\max_{T\in {\cal T}_h}h_T$ the meshsize for the partition
${\cal T}_h$.

For any given integer $k\geq 1$, denote by
$W_k(T)$ the local discrete space of  the weak functions given by
\begin{equation}\label{wk}
   W_k(T)=\{\{v_0, v_b\}: v_0\in P_k(T), v_b\in
P_{k-1}(e), e\subset \partial T\}.
\end{equation}
Patching the local discrete space $W_k(T)$ with a single value on the element interface yields the global finite element space; i.e.,
$$
 W_h=\{v=\{v_0, v_b\}: \ v|_T \in W_k(T), v_b|_{\partial T_i\cap \partial T_j} = v_b|_{\partial T_j\cap \partial T_i}, T, T_i, T_j \in \T_h\},
$$
where $v_b|_{\partial T_i\cap \partial T_j}$ is the value of $v_b$ on $\partial T_i\cap \partial T_j$ as seen from the element $T_i$.
Denote by $W_h^g$ and $W_h^0$ the subspaces of $W_h$ with non-homogeneous and homogeneous boundary values; i.e.,
\begin{equation*}\label{wh0}
\begin{split}
 W_h^g=&\{\{v_0, v_b\} \in W_h:\  v_b|_{e}=Q_bg, e\subset\partial\Omega\},\\
 W_h^0=&\{\{v_0, v_b\} \in W_h:\  v_b|_{e}=0,  e\subset\partial\Omega\},
 \end{split}
\end{equation*}
where $Q_b$ is the $L^2$ projection onto the space $P_{k-1}(e)$.

For $v \in W_h$, denote by $\nabla_{w} v$ the discrete weak action $\nabla_{w, k-1, T} v$ computed by using (\ref{disgradient}) on each element $T$; i.e.,
\begin{equation}\label{wk-1}
 (\nabla_{w} v)|_T= \nabla_{w, k-1, T}(v|_T).
\end{equation}
Next, we introduce the following bilinear forms in $W_h\times W_h$:
\begin{align} \label{EQ:local-stabilizer}
s(u, v)=&\sum_{T\in {\cal T}_h}s_T(u, v),\\
a(u, v)=&\sum_{T\in {\cal T}_h}a_T(u, v),  \label{EQ:local-bterm}
\end{align}
where
\begin{eqnarray*}
 s_T(u, v)&=& h_T^{-1}\langle Q_bu_0-u_b, Q_bv_0-v_b\rangle_{\partial T},\\
 a_T(u, v)&=&(a\nabla_w u, \nabla_w v)_T+(cu_0, v_0)_T.
\end{eqnarray*}

The weak Galerkin finite element scheme for the second order problem (\ref{model})based on the variational formulation (\ref{weakform}) can be stated as follows:
\begin{algorithm}
Find ${\bar u}_h \in  W_{h}^g$, such that
\begin{equation}\label{wg}
s({\bar u}_h, v)+a({\bar u}_h, v)=(f, v_0)\qquad \forall v\in W_h^0.
\end{equation}
\end{algorithm}

\section{Weak Galerkin based on Domain Decompositions}\label{Section:DD-WGFEM}
Let $\{\Omega_j: j=1,\cdots,M\}$ be a partition of $\Omega$ such that
\begin{equation}
\overline{\Omega}=\bigcup_{j=1}^M \overline{\Omega}_j, \qquad \Omega_j\cap\Omega_k=\emptyset, j\neq k.
\end{equation}
In practice, with the exception of a few $\Omega_j$'s along $\partial\Omega$, each $\Omega_j$ is convex with a piecewise-smooth boundary. We introduce
$$
\Gamma=\partial \Omega, \qquad \Gamma_j=\Gamma\cap\partial\Omega_j,
\qquad \Gamma_{jk}=\Gamma_{kj}=\partial\Omega_j\cap\partial\Omega_k.
$$

Assume that the edges/faces of the elements in ${\cal T}_h$ align with the interface $I_f=\bigcup_{j,k=1}^M\Gamma_{jk}$. The partition ${\cal T}_h$ can be grouped into $M$ sets of elements denoted by ${\cal T}_h^{i}={\cal T}_h \cap \Omega_i$, so that each ${\cal T}_h^i$ provides a finite element partition for the subdomain $\Omega_i$ for $i=1, \cdots, M$. The intersection of the partitions ${\cal T}_h^j$ and ${\cal T}_h^k$ also introduces a finite element partition for the interface $\Gamma_{jk}$, denoted by $\Gamma_{jk}^h$. 

Let us introduce the Lagrange multipliers on the edge $\Gamma_{jk}$  as seen from $\Omega_j$ as follows 
$$
\Lambda_{jk}=\{\lambda_{jk}: \lambda_{jk}  \in P_{k-1}(\Gamma_{jk}\cap\Omega_j), \Gamma_{jk}\neq \emptyset \mbox{ is of dimension $d-1$}, \forall j, k=1,\cdots, M\}.
$$
Note that there are two copies of $P_{k-1}(e)$ assigned to the interface $\Gamma_{jk}$ as seen from $\Omega_j$ and  $\Omega_k$ respectively; i.e., $\Lambda_{jk}$ and $\Lambda_{kj}$. Note that the Lagrangian space $\Lambda_{jk}$ is defined only on $d-1$ dimensional interfaces $\Gamma_{jk}\subset\mathbb R^{d-1}$.

Define, for $j=1,\cdots, M$, the finite element space on each subdomain $\Omega_j$:  
\begin{eqnarray*}
W_h(\Omega_j)&=& \{v|_{\Omega_j}: \ v\in W_h\},\\
W_h^0(\Omega_j)&=&\{v\in  W_h(\Omega_j): v|_{\Gamma_j}=0\}.
\end{eqnarray*}

The weak Galerkin finite element method \eqref{wg}  restricted in the subdomain $\Omega_j$ $(j=1, \cdots, M)$ is as follows:  Find $u_j=\{u_{j,0}, u_{j,b}\}\in W_h(\Omega_j)$, such that $u_{j, b}=Q_b g$ on $\Gamma_j$, $\lambda_{jk}\in \Lambda_{jk}$, $j, k=1, \cdots, M$,
\begin{equation}\label{DD-Alg}
\left\{
\begin{split}
&(a\nabla_w u_j, \nabla_w v_j)_{\Omega_j}+s_{j}(u_j, v_j)-\sum_{k=1}^M\langle \lambda_{jk}, v_{j, b}\rangle_{\Gamma_{jk}}\\
&+(cu_{j,0}, v_{j,0})_{\Omega_j} =(f, v_{j,0})_{\Omega_j}, \qquad \forall v_j\in W_h^0(\Omega_j),\\
&\sum_{k=1}^M \langle \mu, \ljump u_b\rjump_{\Gamma_{jk}} \rangle_{\Gamma_{jk}} =0, \qquad\forall \mu\in \Lambda_{jk},\\
&\lambda_{jk}+\lambda_{kj}=0, \qquad \text{on}\  \Gamma_{jk},
\end{split}
\right.
\end{equation}
where $(\cdot, \cdot)_{\Omega_j}=\sum_{T\in {\cal T}_h^j} (\cdot, \cdot)_T$,
$s_{j}(u_j, v_j)=\sum_{T\in {\cal T}_h^j}s_T(u,v)$, and $\ljump u_b\rjump_{\Gamma_{jk}}$ is the jump of $u_{b}$ on $\Gamma_{jk}$ defined as follows:
\begin{equation}\label{jump}
  \ljump u_b\rjump_{\Gamma_{jk}}=
   u_{b,j}|_{\Gamma_{jk}\cap\Omega_{j}}-u_{b,k}|_{\Gamma_{kj}\cap\Omega_{k}},
 \end{equation}
 where $u_{b,j}|_{\Gamma_{jk}\cap\Omega_{j}}$ and $u_{b,k}|_{\Gamma_{kj}\cap\Omega_{k}}$ represent the values of $u_b$ on $\Gamma_{jk}$ as seen from $\Omega_{j}$ and $\Omega_{k}$ respectively.

\begin{lemma}
Let $u_j=\{u_{j,0}, u_{j,b}\}\in {W}_h(\Omega_j)$ and $\lambda_{jk} \in \Lambda_{jk}$ be the solution of the algorithm \eqref{DD-Alg}. Then, we have $\ljump u_b\rjump_{\Gamma_{jk}}=0$ for $j, k=1, \cdots, M$, so that $u_h|_{\Omega_j}=u_j$ is a function in the finite element space $W_h$.  Furthermore, this function $u_h$ satisfies the WG scheme (\ref{wg}). In other words, we have $u_h\equiv {\bar u}_h$.
\end{lemma}
\begin{proof}
By letting $\mu=\ljump u_b\rjump_{\Gamma_{jk}}\in \Lambda_{jk}$ and then using the second equation in \eqref{DD-Alg} we arrive at
$$
0=\sum_{j, k=1}^M\int_{\Gamma_{jk}} \ljump u_b\rjump_{\Gamma_{jk}}^2ds,
$$
which implies that $\ljump u_b\rjump_{\Gamma_{jk}}=0$ for $j, k=1, \cdots, M$.

Next, by restricting $v\in {W}_h^0$ in the first equation in \eqref{DD-Alg}
 and using the fact that $\lambda_{jk}+\lambda_{kj}=0$ we have
$$
\sum_{j, k=1}^M\langle \lambda_{jk}, v_{b}\rangle_{\Gamma_{jk}}=0,
$$ 
which leads to
$$
 a(u_h, v)+s(u_h, v)=(f, v_{0})\qquad \forall v\in W_h^0.
 $$
It follows from $u_{j, b}=Q_b g$ on $\Gamma_j$ that 
$u_h \in W_h^g$, which together with the last equation shows that $u_h$ is a solution of the WG scheme \eqref{wg}. Finally, from the solution uniqueness for \eqref{wg} we have $u_h\equiv {\bar u}_h$. This completes the proof.
\end{proof}

\section{A Parallel Iterative Scheme}\label{Section:IterativeScheme}
A parallel iterative scheme for the weak Galerkin finite element method can be designed by using the equivalent numerical scheme \eqref{DD-Alg}. The motivation of the iterative procedure comes from the observation that, in \eqref{DD-Alg}, the following consistency conditions 
\begin{equation}
\begin{split}
u_{j, b}&=u_{k, b},\qquad \text{on}\ \Gamma_{jk},\\
\lambda_{jk}&=-\lambda_{kj},\qquad \text{on}\  \Gamma_{jk},
\end{split}
\end{equation}
are equivalent to
\begin{equation}
\begin{split}
\beta u_{j, b}+\lambda_{jk}&=\beta u_{k, b}-\lambda_{kj},\qquad \text{on}\ \Gamma_{jk},\\
\beta u_{k, b}+\lambda_{kj}&=\beta u_{j, b}-\lambda_{jk},\qquad \text{on}\  \Gamma_{kj},
\end{split}
\end{equation}
for any non-zero function $\beta$ on $\bigcup_{j, k=1}^M\Gamma_{jk}$. This gives rise to
\begin{equation}\label{EQ:101}
\langle \lambda_{jk}, v_{j,b}\rangle_{\Gamma_{jk}}=\langle \beta(u_{k,b}-u_{j,b})-\lambda_{kj}, v_{j,b}\rangle_{\Gamma_{jk}},
\end{equation}
where $v_j=\{v_{j,0}, v_{j,b}\} \in W_h(\Omega_j)$. The use of the equation \eqref{EQ:101} in the numerical scheme \eqref{DD-Alg} is a critical step in the design of the following parallel iterative procedure.

Starting from any initial guess $u_j^{(0)}=\{u_{j,0}^{(0)}, u_{j,b}^{(0)}\}\in W_h(\Omega_j)$ with $u_{j, b}^{(0)}=Q_bg$ on $\Gamma_j$, $\lambda^{(0)}_{jk}\in \Lambda_{jk}$ and $\lambda^{(0)}_{kj}\in \Lambda_{kj}$, we solve for $u_j^{(n)}=\{u_{j,0}^{(n)}, u_{j,b}^{(n)}\}\in W_h(\Omega_j)$, $\lambda^{(n)}_{jk}\in \Lambda_{jk}$, and $\lambda^{(n)}_{kj}\in \Lambda_{kj}$ such that $u_{j, b}^{(n)}=Q_bg$ on $\Gamma_j$ and satisfying the following system of linear equations
\begin{equation}\label{iter}
\left\{
\begin{split}
&(a\nabla_w u_j^{(n)}, \nabla_w v_j)_{\Omega_j}+s_{j}(u_j^{(n)}, v_j)+\sum_{k=1}^M\langle \beta u_{j, b}^{(n)}, v_{j, b}\rangle_{\Gamma_{jk}}+(cu_{j, 0}^{(n)}, v_{j, 0})_{\Omega_j}\\
&=\sum_{k=1}^M \langle \beta u_{k, b}^{(n-1)}-\lambda_{kj}^{(n-1)}, v_{j, b}\rangle_{\Gamma_{jk}}+(f, v_{j, 0})_{\Omega_j}\quad   \forall v_j\in W_h^0(\Omega_j), \\
&\lambda_{jk}^{(n)}=\beta (u_{k, b}^{(n-1)}-u_{j, b}^{(n)})-\lambda_{kj}^{(n-1)}.
\end{split}\right.
\end{equation}

\section{Convergence}\label{Section:Convergence}
Let us establish the convergence for the iteration procedure defined in \eqref{iter}. To this end, for $e=\{e_j: \ \ e_j\in W_h(\Omega_j)\}$ and $\mu=\{\mu_{jk}:\ \mu_{jk}\in \Lambda_{jk}\}$, denote by 
\begin{eqnarray*}
(a\nabla_w e, \nabla_w e)&=&\sum_{j=1}^M (a\nabla_w e_j, \nabla_w e_j)_{\Omega_j},\\
(ce_{0}, e_{0})&=&\sum_{j=1}^M(ce_{j,0}, e_{j,0})_{\Omega_j},\\
s(e, e)&=&\sum_{j=1}^M s_j(e_j, e_{j}).
\end{eqnarray*}
Moreover, denote by
\begin{eqnarray*}
E(\{e, \mu\})&=&\sum_{j=1}^M \beta^2 |e_{j, b}|_{0, B_j}^2+\sum_{j,k=1}^M|\mu_{jk}|_{0, \Gamma_{jk}}^2\\
& & +2\beta\{(a\nabla_w e, \nabla_w e)+s(e, e)+(ce_{0}, e_{0})\},
\end{eqnarray*}
where $B_j=\partial\Omega_j\setminus \Gamma_j$.

Let $\{u_j, \lambda_{jk}\}$ be the solution of the domain-decomposition-based numerical scheme \eqref{DD-Alg}. We define the error functions:
\begin{equation}\label{EQ:21}
e_j^{(n)}=u_j-u_j^{(n)}, \qquad \mu_{jk}=\lambda_{jk}-\lambda_{jk}^{(n)}, \qquad \mu_{kj}=\lambda_{kj}-\lambda_{kj}^{(n)}.
\end{equation}
The error equations for the iterative procedure \eqref{iter} can be written in the form:
\begin{equation}\label{error}
\left\{
\begin{split}
&(a\nabla_w e_j^{(n)}, \nabla_w v_j)_{\Omega_j}+s_{j}(e_j^{(n)}, v_j)-\sum_{k=1}^M\langle \mu_{jk}^{(n)}, v_{j, b}\rangle_{\Gamma_{jk}}\\&+(ce_{j, 0}^{(n)}, v_{j, 0})_{\Omega_j}= 0, \qquad \forall v_j\in W_h(\Omega_j),\\
&\mu_{jk}^{(n)}=\beta(e_{k, b}^{(n-1)}-e_{j, b}^{(n)})-\mu_{kj}^{(n-1)},
\end{split}\right.
\end{equation} 

\begin{lemma}\label{Lemma:6-1}
Let  $E^{(n)}=E(\{e^{(n)}, \mu^{(n)}\})$. The following identities hold true
\begin{equation}\label{en:10}
\begin{split}
& E^{(n-1)}-E^{(n)} \\
=&4\beta\{(a\nabla_w e^{(n-1)}, \nabla_w e^{(n-1)})+s(e^{(n-1)}, e^{(n-1)})+(ce_{0}^{(n-1)}, e_{0}^{(n-1)})\},
\end{split}
\end{equation}
and
\begin{equation}\label{en:08}
\begin{split}
E^{(n)}=&\sum_{k=1}^M|\beta e_{j, b}^{(n)}+\mu_{jk}^{(n)}|^2_{0, \Gamma_{jk}}\\
=&\beta^2 |e_{j, b}^{(n)}|_{0, B_j}^2+\sum_{k=1}^M|\mu_{jk}^{(n)}|_{0, \Gamma_{jk}}^2 + 2\beta \sum_{k=1}^M \langle e_{j, b}^{(n)}, \mu_{jk}^{(n)} \rangle_{\Gamma_{jk}}\\
=& \beta^2 |e_{j, b}^{(n)}|_{0, B_j}^2+\sum_{k=1}^M|\mu_{jk}^{(n)}|_{0, \Gamma_{jk}}^2\\
&+2\beta\{(a\nabla_w e_j^{(n)}, \nabla_w e_j^{(n)})_{\Omega_j}+s_{j}(e_j^{(n)}, e_j^{(n)})+(ce_{j, 0}^{(n)}, e_{j, 0}^{(n)})_{\Omega_j}\}.
\end{split}
\end{equation}
\end{lemma}

\begin{proof} Note that $e_{j,b}^{(n)}=0$ on $\Gamma_j$. 
By letting $v_j=e_j^{(n)}$ in the first equation of \eqref{error} we obtain
\begin{equation}\label{t1_1}
(a\nabla_w e_j^{(n)}, \nabla_w e_j^{(n)})_{\Omega_j}+s_{j}(e_j^{(n)}, e_j^{(n)})-\sum_{k=1}^M \langle \mu_{jk}^{(n)}, e_{j, b}^{(n)}\rangle_{\Gamma_{jk}}+(ce_{j, 0}^{(n)}, e_{j, 0}^{(n)})_{\Omega_j}=0.
\end{equation}
 From \eqref{t1_1} we have
\begin{equation*}
\begin{split}
E^{(n)}=&\sum_{k=1}^M|\beta e_{j, b}^{(n)}+\mu_{jk}^{(n)}|^2_{0, \Gamma_{jk}}\\
=&\beta^2 |e_{j, b}^{(n)}|_{0, B_j}^2+\sum_{k=1}^M|\mu_{jk}^{(n)}|_{0, \Gamma_{jk}}^2 + 2\beta \sum_{k=1}^M \langle e_{j, b}^{(n)}, \mu_{jk}^{(n)} \rangle_{\Gamma_{jk}}\\
=& \beta^2 |e_{j, b}^{(n)}|_{0, B_j}^2+\sum_{k=1}^M|\mu_{jk}^{(n)}|_{0, \Gamma_{jk}}^2\\
&+2\beta\{(a\nabla_w e_j^{(n)}, \nabla_w e_j^{(n)})_{\Omega_j}+s_{j}(e_j^{(n)}, e_j^{(n)})+(ce_{j, 0}^{(n)}, e_{j, 0}^{(n)})_{\Omega_j}\},
\end{split}
\end{equation*}
which verifies the identify \eqref{en:08}.

Next, from the second equation in \eqref{error} we have
\begin{equation}\label{en}
\begin{split}
E^{(n)}=&\sum_{j=1}^M\sum_{k=1}^M|\beta e_{j, b}^{(n)} + \mu_{jk}^{(n)}|^2_{0, \Gamma_{jk}}\\
=&\sum_{k=1}^M\sum_{j=1}^M|\beta e_{k, b}^{(n-1)}-\mu_{kj}^{(n-1)}|^2_{0, \Gamma_{kj}}\\
=& E^{(n-1)}-4\beta\sum_{k=1}^M\{(a\nabla_w e_k^{(n-1)}, \nabla_w e_k^{(n-1)})_{\Omega_k}+s_{k}(e_k^{(n-1)}, e_k^{(n-1)})\\
&+(ce_{k, 0}^{(n-1)}, e_{k, 0}^{(n-1)})_{\Omega_k}\},
\end{split}
\end{equation}
which leads to the identity \eqref{en:10}.
\end{proof}

\begin{theorem}
Let $\{u_j^{(n)}, \lambda_{jk}^{(n)}\}$ be the solution of the iterative scheme \eqref{iter} and $\{e_j^{(n)}, \mu_{jk}^{(n)}\}$ be the error functions defined in \eqref{EQ:21} for the numerical scheme \eqref{DD-Alg}. For any $\beta >0$, the following convergence holds true:
\begin{eqnarray}
u_j^{(n)}\to u_j,\label{EQ:e-conv}\\
\lambda_{jk}^{(n)}\to \lambda_{jk} \label{EQ:mu-conv}
\end{eqnarray}
as $n\to \infty$.
\end{theorem}

\begin{proof}
Note that, from Lemma \ref{Lemma:6-1}, $\{E^{(n)}\}$ is a decreasing sequence of nonnegative numbers. Thus, we have
\begin{equation}
\begin{split}
\sum_{n=0}^\infty \{(a\nabla_w e^{(n)}, \nabla_w e^{(n)})+s(e^{(n)}, e^{(n)})+(ce_{0}^{(n)}, e_{0}^{(n)}) \} <\infty.
\end{split}
\end{equation}
Note that $a>0$.  If $c>0$, it is easy to obtain $e_{0}^{(n)} \to 0$ on each subdomain $\Omega_j$, which, using $s(e^{(n)}, e^{(n)})\to 0$, gives $e_b^{(n)}\to 0$ and further $e^{(n)}\to 0$ as $n\to \infty$. 
For the case of $c\ge 0$, the above argument would not go through so that new approaches are necessary. The rest of the proof assumes the general case of $c\ge 0$.

For $\mu^{(n)}=\{\mu_{jk}^{(n)}\}$, we construct $v^* = \{v^*_j\}$, where $v_j^*\in W_h(\Omega_j)$ assumes the value of $\mu_{jk}^{(n)}$ on $\Gamma_{jk}$ and zero otherwise. It is easy to show that 
\begin{equation}\label{error-02}
\|\nabla_w v^*_j\|_{\Omega_j}^2 + s_j(v^*_j, v^*_j) \leq C h^{-1} \sum_{k=1}^M \|\mu_{jk}^{(n)}\|_{\Gamma_{jk}}^2.
\end{equation}
Substituting $v_j$ by $v_j^*$ in \eqref{error} yields
\begin{equation}\label{error-01}
(a\nabla_w e_j^{(n)}, \nabla_w v_j^*)_{\Omega_j}+s_{j}(e_j^{(n)}, v_j^*)-\sum_{k=1}^M\|\mu_{jk}^{(n)}\|_{\Gamma_{jk}}^2 = 0.
\end{equation}
Hence,
\begin{equation*}
\begin{split}
\sum_{k=1}^M\|\mu_{jk}^{(n)}\|_{\Gamma_{jk}}^2 = & (a\nabla_w e_j^{(n)}, \nabla_w v_j^*)_{\Omega_j}+s_{j}(e_j^{(n)}, v_j^*) \\  
\leq &C (a\nabla_w e_j^{(n)}, \nabla_w e_j^{(n)})^{1/2}_{\Omega_j} \|\nabla_w v^*_j\|_{\Omega_j} + s_{j}(e_j^{(n)}, e_j^{(n)})^{1/2} s_j(v^*_j, v^*_j)^{1/2}.
\end{split}
\end{equation*}
Using the estimate \eqref{error-02} in the above inequality gives
\begin{equation*}
\begin{split}
\sum_{k=1}^M\|\mu_{jk}^{(n)}\|_{\Gamma_{jk}}^2 = & (a\nabla_w e_j^{(n)}, \nabla_w v_j^*)_{\Omega_j}+s_{j}(e_j^{(n)}, v_j^*) \\  
\leq &Ch^{-1}\{ (a\nabla_w e_j^{(n)}, \nabla_w e_j^{(n)}) + s_{j}(e_j^{(n)}, e_j^{(n)})\}.
\end{split}
\end{equation*}
It follows that $\sum_{j,k=1}^M\|\mu_{jk}^{(n)}\|_{\Gamma_{jk}}^2 \to 0$ as $n\to \infty$, which asserts the convergence for the Lagrangian multiplier $\lambda_{jk}^{(n)}$.

To prove the convergence of $u_j^{(n)}$, we note from \eqref{en:08} that $\{e_{j,b}^{(n)}\}$ is a bounded sequence in $L^2(B_j)$, and hence from Lemma \ref{lemma:6-4}, $\{e_{j}^{(n)}\}$ is a bounded sequence with respect to the norm $\3bar\cdot\3bar_{1,\Omega_j}$. It follows that $\{e_{j}^{(n)}\}$ must have a convergent subsequence. Without loss of generality, we may assume that $\{e_{j}^{(n)}\}$ is convergent so that
$$
e_{j}^{(n)} \to e^*_j\quad \mbox{as $n\to \infty$.}
$$
By passing to the limit of $n\to\infty$ and using the face that $\mu_{jk}^{(n)} \to 0$, from \eqref{error} we obtain the following equations:
\begin{equation}\label{error-88}
\begin{split}
&(a\nabla_w e_j^*, \nabla_w v_j)_{\Omega_j}+s_{j}(e_j^*, v_j)+(ce_{j, 0}^*, v_{j, 0})_{\Omega_j}= 0, \qquad \forall v_j\in W_h(\Omega_j),\\
&e_{j, b}^*=e_{k, b}^*,\quad \mbox{on $\Gamma_{jk}$},
\end{split}
\end{equation} 
which implies that $e_j^* \equiv 0$ for $j=1,\cdots, M$. This completes the proof of the theorem.
\end{proof}

The rest of this sections is concerned with two technical results that support the convergence analysis for the parallel iterative procedure.

\begin{lemma}\label{lemma:6-4}
\label{equiv}There exists $C_1>0$ and $C_2>0$ such that
\begin{eqnarray*}
C_1((\nabla v_{0}, \nabla v_{0})_{\Omega_j}+s_j(v, v))\leq  (\nabla_w v, \nabla_w v)_{\Omega_j}+s_j(v, v)\\
\leq C_2((\nabla v_{0}, \nabla v_{0})+s_j(v, v))
\end{eqnarray*}
for any $v\in W_h(\Omega_j)$.
\end{lemma}

\begin{proof}
For each $T\in \T_h^j$, using \eqref{disgradient} and the usual integration by parts yields
\begin{equation*}
\begin{split}
(\nabla_w v, \bw)_T=&-(v_0, \nabla \cdot\bw)_T+\langle v_b, \bw\cdot\bn\rangle_{\partial T}\\
=&(\nabla v_0, \bw)_T+\langle v_b - Q_b v_0, \bw\cdot\bn\rangle_{\partial T}.
\end{split}
\end{equation*}
Now from the Cauchy-Schwarz inequality and the trace inequality we obtain
\begin{equation*}
\begin{split}
\|\nabla_w v\|_T\leq &\frac{\|\nabla v_0\|_T \|\bw\|_T+\|Q_bv_0-v_b\|_{\pT} \|\bw\cdot\bn\|_{\partial T}}{\|\bw\|_T}\\
\leq &C_2\frac{\|\nabla v_0\|_T \|\bw\|_T+\|Q_bv_0-v_b\|_{\pT} h_T^{-\frac{1}{2}}\|\bw\|_{T}}{\|\bw\|_T}\\
\leq& C_2(\|\nabla v_0\|_T+h_T^{-\frac{1}{2}}\|Q_bv_0-v_b\|_{\pT}).
\end{split}
\end{equation*}
Summing over all $T\in \T_h^j$ gives rise to the upper-bound estimate of $(\nabla_w v, \nabla_w v)_{\Omega_j}+s_j(v, v)$. 
The lower-bound estimate of $(\nabla_w v, \nabla_w v)_{\Omega_j}+s_j(v, v)$ can be established analogously.
This completes the proof of the lemma.
\end{proof}

For $v_j=\{v_{j,0}, v_{j,b}\}\in {W}_h(\Omega_j)$, we define a semi-norm by setting
\begin{equation}\label{norm}
\3bar v_j \3bar^2_{1,\Omega_j}=\sum_{T\in \T_h^j}(a\nabla v_{j,0}, \nabla  v_{j,0})_T +s_j(v_j, v_j) + \|v_{j,b}\|^2_{\partial\Omega_j}.
\end{equation}

\begin{lemma}\label{lemnorm} The semi-norm $\3bar \cdot \3bar_{1,\Omega_j}$ defined in \eqref{norm} is a norm in ${W}_h(\Omega_j)$.
\end{lemma}

\begin{proof}
It suffices to verify the positivity property for $\3bar \cdot \3bar_{1,\Omega_j}$. To this end, assume $\3bar v_j \3bar_{1,\Omega_j}=0$ for a weak function $v=\{v_{j,0}, v_{j,b}\}\in {W}_h(\Omega_j)$. It follows that $\nabla v_{j,0}=0$ on each $T\in {\cal T}_h$, $Q_bv_{j,0}=v_{j,b}$ on each $\partial T$, and $v_{j,b}=0$ on $\partial\Omega_j$. Therefore, we have $v_{j,0}=const$ on each $T\in {\cal T}_h^j$. Using $Q_bv_{j,0}=v_b$ on $\partial T$, we have $v_{j,0}=const$ in the subdomain $\Omega_j$. Using $v_{j,b}=0$ on $\partial \Omega_j$ and $Q_bv_{j,0}=v_{j,b}$ on $\partial T$, we obtain $v_{j,0}=0$ in $\Omega_j$. Again from $Q_bv_{j,0}=v_{j,b}$ on $\partial T$ we have $v_{j,b}=0$ in $\Omega_j$. This completes the proof of the lemma.
\end{proof}

\section{Numerical Experiments}\label{Section:numerical-experiments}

\def\ad#1{\begin{aligned}#1\end{aligned}}  \def\b#1{\mathbf{#1}} \def\t#1{\text{#1}}
\def\a#1{\begin{align*}#1\end{align*}} \def\an#1{\begin{align}#1\end{align}}
\def\t#1{\operatorname{#1}}

In this section, we shall report some numerical results to demonstrate the performance of the iterative procedure for the weak Galerkin finite element method \eqref{wg} for the second order elliptic model problem \eqref{model}.

\subsection{Test Example 1} The configuration of the test is set up as follows:  the coefficients are $a(x,y)=2-x(1-x)$ and $c=1$; the exact solution is  \an{\label{s-1} u(x,y)=2^6 x^2 (1-x)^2 y^2(1-y)^2;  }
   and the domain is  $\Omega=(0,1)^2$. Note that this corresponds to a homogeneous Dirichlet boundary value problem.
The triangular grids  shown as in Figures \ref{grid1} and \ref{grid2} are employed in the numerical tests. The 4-subdomain iterations and 16-subdomain iterations are computed respectively.
In Tables \ref{t01}-\ref{t06}, we list the computational  results
  for the $\{P_k,P_{k-1}\}$ weak Galerkin finite element defined in \eqref{wk} when
 the weak gradient is discretized by $[P_{k-1}(T)]^2$ defined in \eqref{wk-1}. Note that $Q_h u=\{ Q_0 u, Q_b u \}$  where $Q_0$ is the element-wise $L^2$ projection to space $P_k(T)$ and $Q_b$ is the edge-wise $L^2$ projection to $P_{k-1}(e)$.
In the computation,  the iteration stops when the iterative error reaches the truncation error.

\begin{figure}[ht] \begin{center}
\begin{picture}(320,100)(0,0)
\put(0,-50){\includegraphics[width=1.6in]{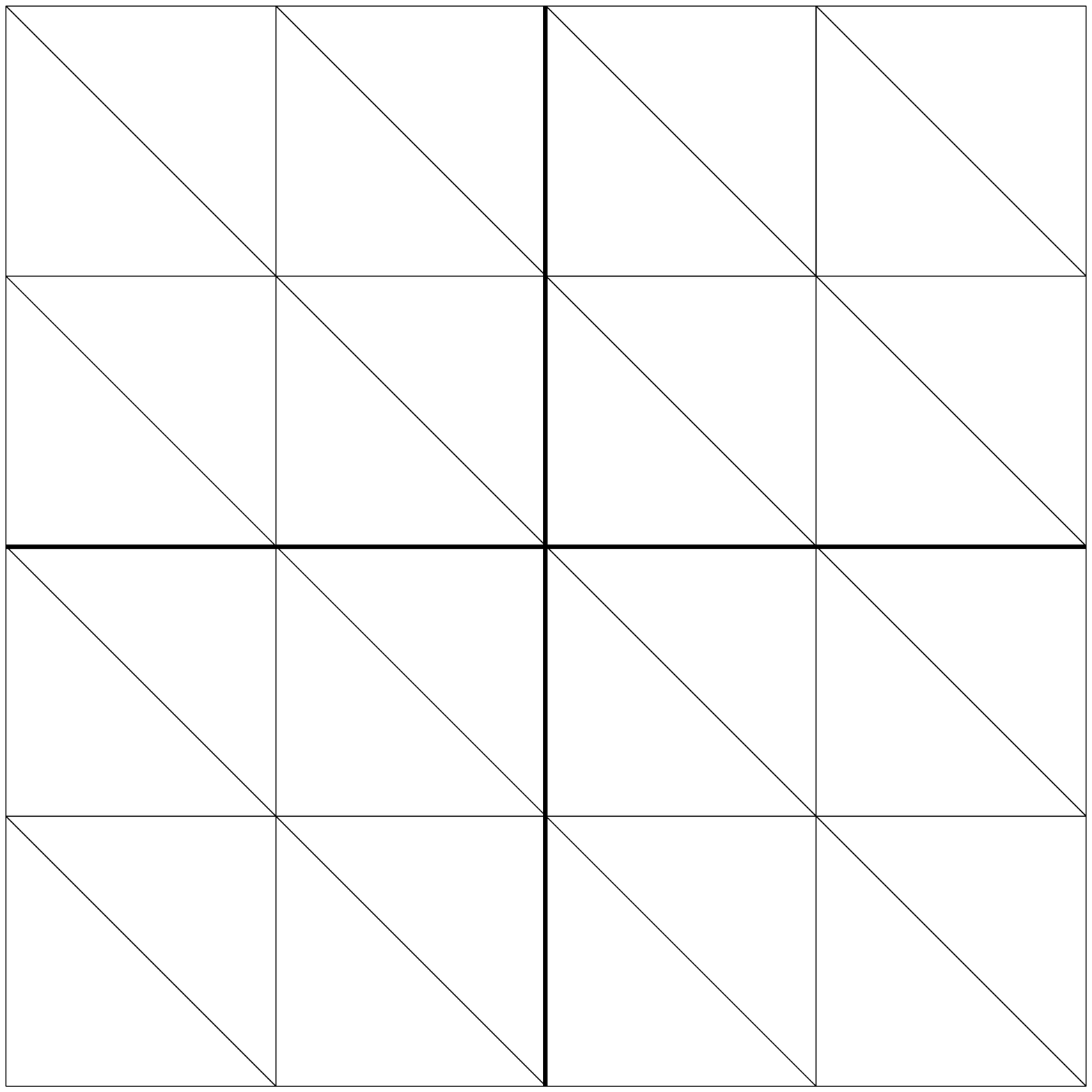}}
\put(110,-50){\includegraphics[width=1.6in]{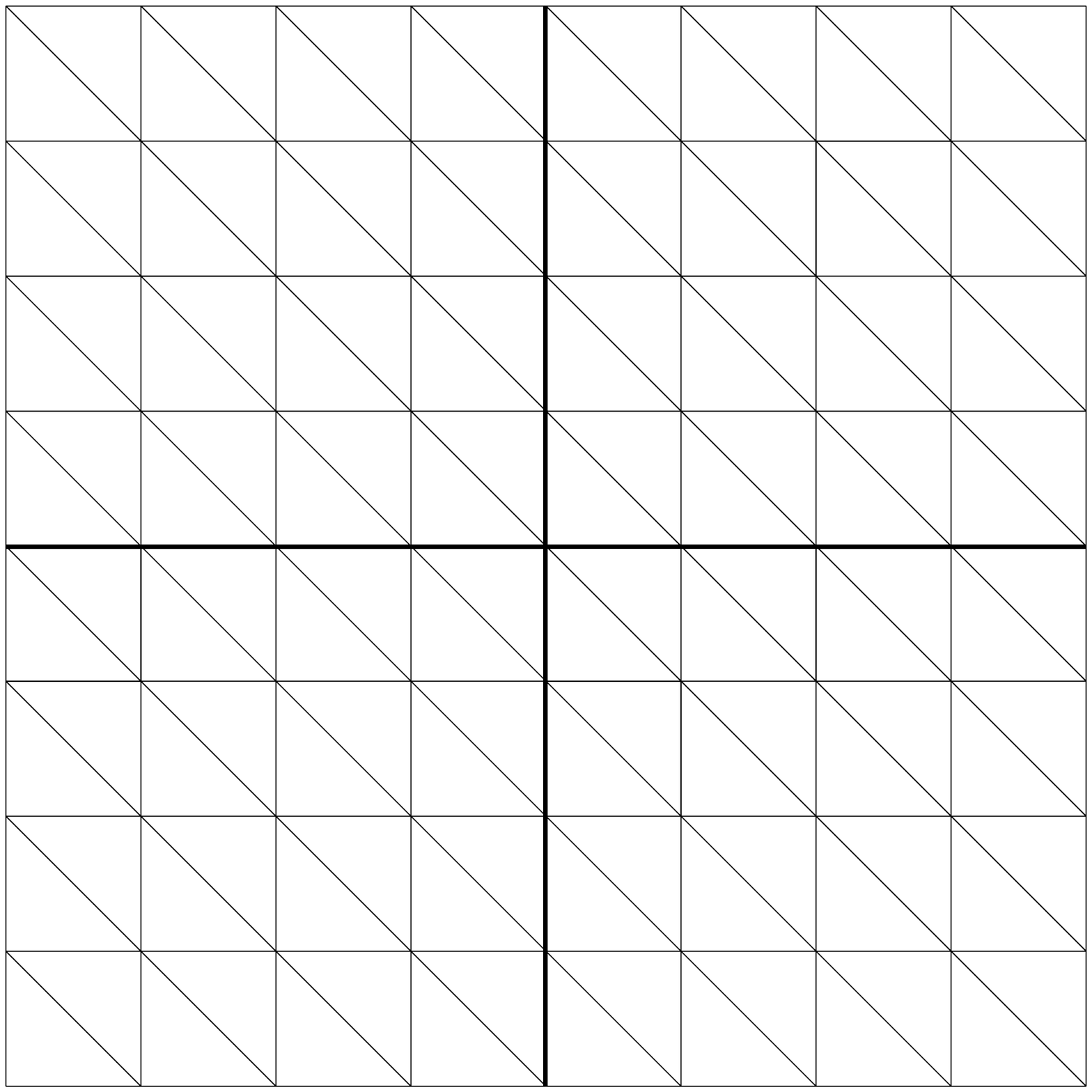}}
\put(220,-50){\includegraphics[width=1.6in]{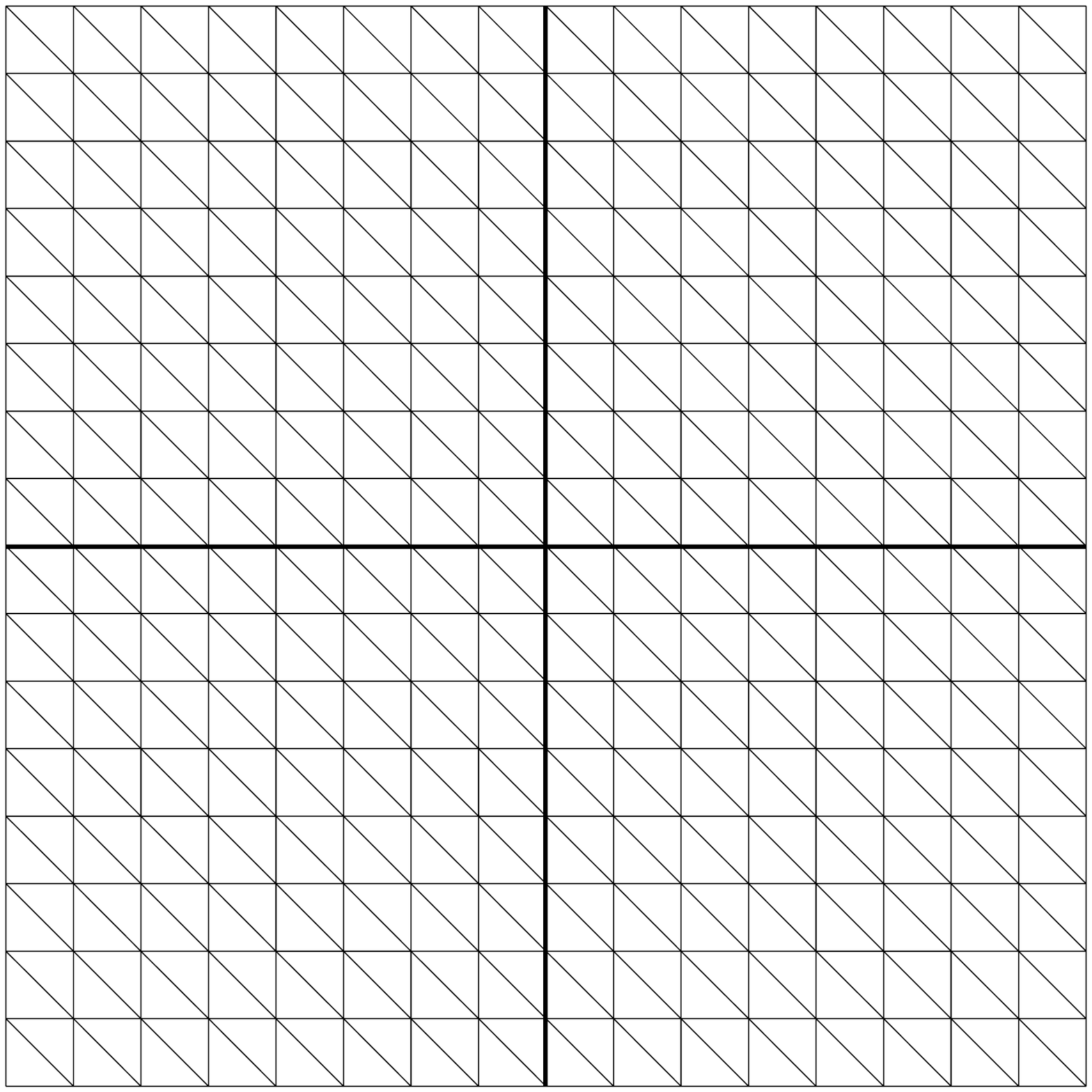}} \end{picture}
\caption{ The level 2, 3 and 4 grids for the 4-subdomain computation.  }
\label{grid1} \end{center}
\end{figure}

\begin{figure}[ht] \begin{center}
\begin{picture}(320,110)(0,0)
\put(0,-50){\includegraphics[width=1.6in]{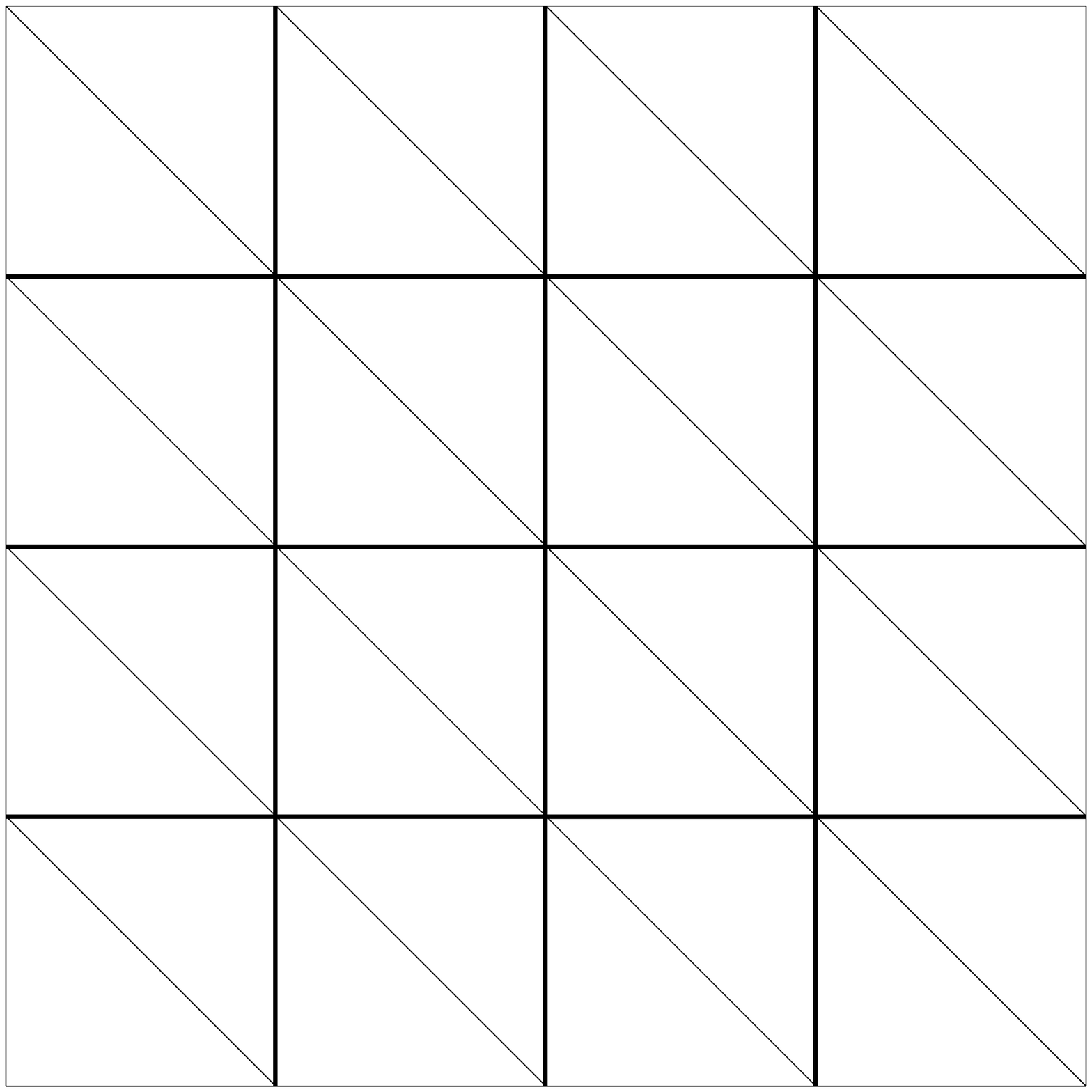}}
\put(110,-50){\includegraphics[width=1.6in]{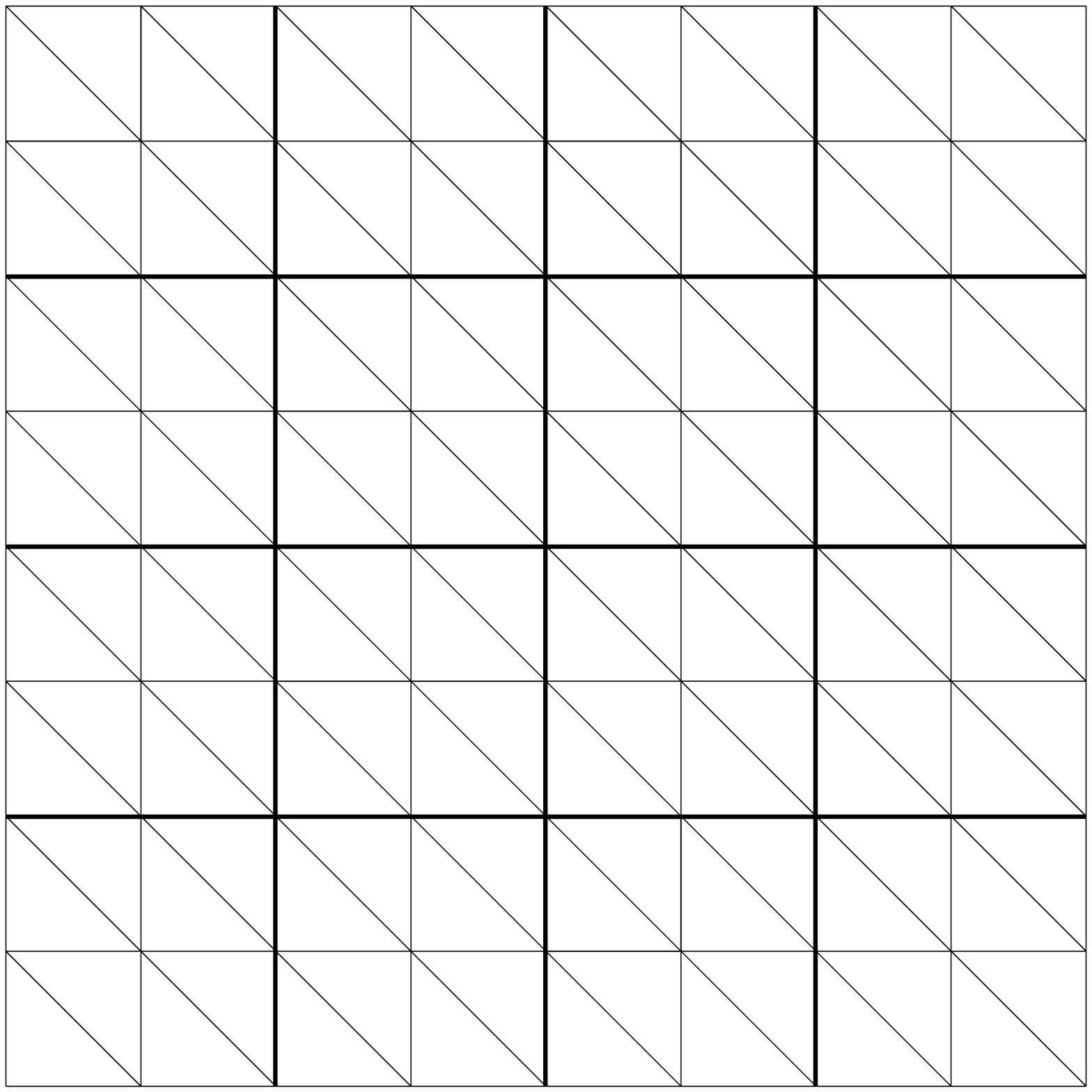}}
\put(220,-50){\includegraphics[width=1.6in]{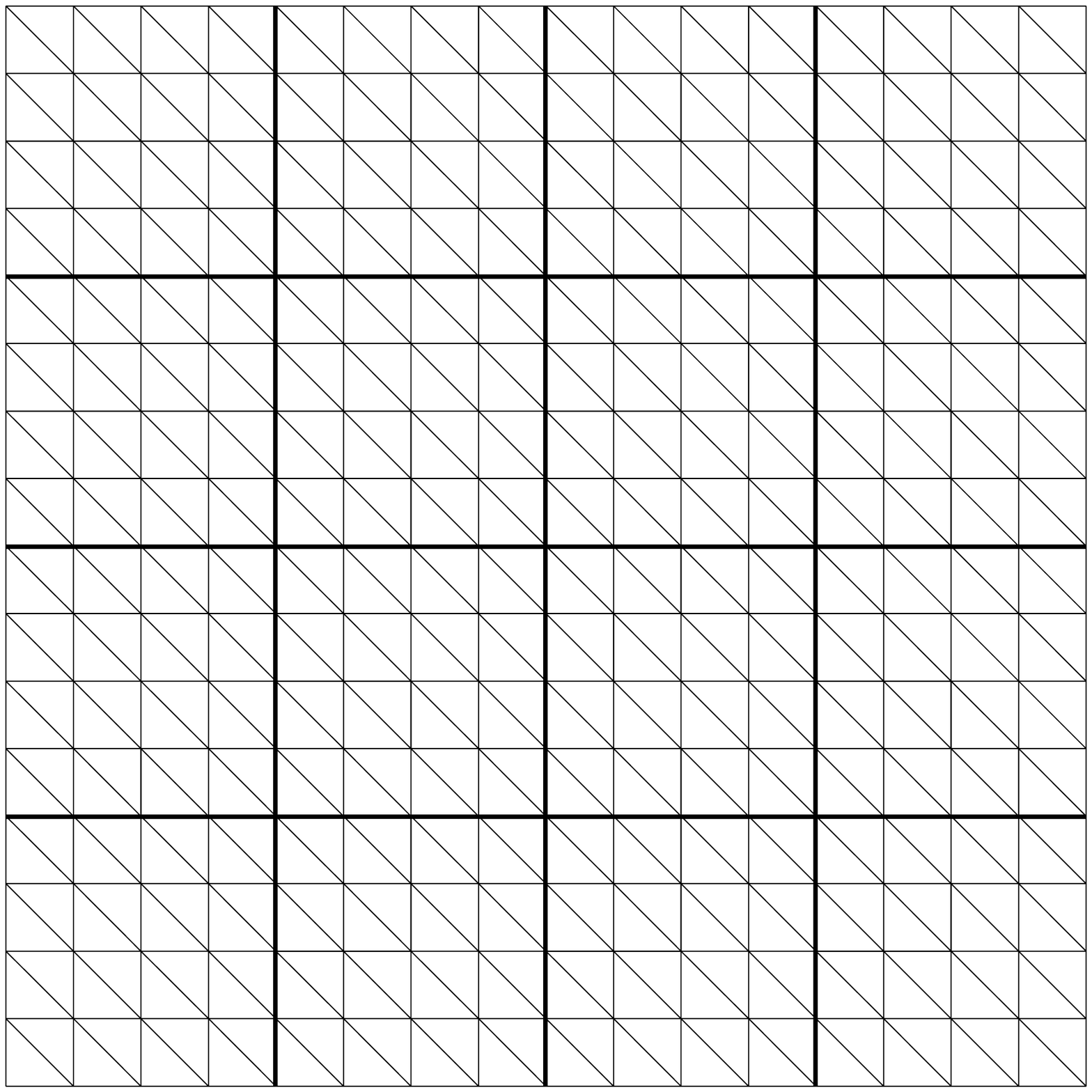}} \end{picture}
\caption{ The level 2, 3 and 4 grids for the 16-subdomain computation.  }
\label{grid2} \end{center}
\end{figure}

\begin{table}[ht]
  \centering   \renewcommand{\arraystretch}{1.05}
  \caption{Error profile of $\{P_1,P_0\}$ solutions for \eqref{s-1} on
      Figures \ref{grid1}--\ref{grid2} grids. }
\label{t01}
\begin{tabular}{c|cc|cc|r}
\hline
Grid & $\|Q_0 u - u_0 \|_0 $ &rate & $\|\nabla_w (Q_h u - u_h)\|_0$ &rate &\# iteration   \\
\hline
 &\multicolumn{5}{c}{The $\{P_1,P_0\}$ WG element \eqref{wk}, by 4-subdomain iteration.} \\
 \hline
 1&    0.211E+00 &  0.0&    0.614E-01 &  0.0 &      6 \\
 2&    0.684E-01 &  1.6&    0.181E+00 &  0.0 &      7 \\
 3&    0.178E-01 &  1.9&    0.110E+00 &  0.7 &      9 \\
 4&    0.448E-02 &  2.0&    0.579E-01 &  0.9 &     11 \\
 5&    0.112E-02 &  2.0&    0.293E-01 &  1.0 &     11 \\
 6&    0.280E-03 &  2.0&    0.147E-01 &  1.0 &     13 \\
 7&    0.692E-04 &  2.0&    0.734E-02 &  1.0 &     13 \\
\hline
 &\multicolumn{5}{c}{The $\{P_1,P_0\}$ WG element \eqref{wk}, by 16-subdomain iteration.} \\
 \hline
 1&    0.211E+00 &  0.0&    0.614E-01 &  0.0 &      6 \\
 2&    0.684E-01 &  1.6&    0.181E+00 &  0.0 &      7 \\
 3&    0.178E-01 &  1.9&    0.110E+00 &  0.7 &      9 \\
 4&    0.448E-02 &  2.0&    0.578E-01 &  0.9 &     11 \\
 5&    0.112E-02 &  2.0&    0.292E-01 &  1.0 &     11 \\
 6&    0.279E-03 &  2.0&    0.147E-01 &  1.0 &     13 \\
 7&    0.680E-04 &  2.0&    0.734E-02 &  1.0 &     13 \\
\hline
\end{tabular}%
\end{table}%

In Table \ref{t01}, we choose $\beta=8$ and the $\{P_1,P_0\}$ WG finite element \eqref{wk}. We can see from  Table \ref{t01} that we have an order two convergence in the  $L^2$-norm and an order one convergence in the energy norm.
We list the number of domain-decomposition iterations when the domain is subdivided into 4 subdomains and 16 subdomains respectively. In theory, the number of iterations may increase on higher level grids.
However, the number of iterations appears to be steady when using 4 subdomains and 16 subdomains.

\begin{table}[ht]
  \centering   \renewcommand{\arraystretch}{1.05}
  \caption{Error profile of $\{P_2,P_1\}$ solutions  for \eqref{s-1}  on
      Figures \ref{grid1}--\ref{grid2} grids. }
\label{t02}
\begin{tabular}{c|cc|cc|r}
\hline
Grid & $\|Q_0 u - u_0 \|_0 $ &rate & $\|\nabla_w (Q_h u - u_h)\|_0$ &rate &\# iteration   \\
\hline
 &\multicolumn{5}{c}{The $\{P_2,P_1\}$ WG element \eqref{wk}, by 4-subdomain iteration.} \\
 \hline
 1&    0.106E+00 &  0.0&    0.110E+00 &  0.0 &      6 \\
 2&    0.161E-01 &  2.7&    0.351E-01 &  1.7 &      9 \\
 3&    0.208E-02 &  3.0&    0.102E-01 &  1.8 &      9 \\
 4&    0.261E-03 &  3.0&    0.274E-02 &  1.9 &     11 \\
 5&    0.327E-04 &  3.0&    0.708E-03 &  2.0 &     17 \\
 6&    0.411E-05 &  3.0&    0.183E-03 &  2.0 &     20 \\
\hline
 &\multicolumn{5}{c}{The $\{P_2,P_1\}$ WG element \eqref{wk}, by 16-subdomain iteration.} \\
 \hline
 1&    0.106E+00 &  0.0&    0.110E+00 &  0.0 &      6 \\
 2&    0.161E-01 &  2.7&    0.352E-01 &  1.6 &      8 \\
 3&    0.207E-02 &  3.0&    0.103E-01 &  1.8 &     10 \\
 4&    0.262E-03 &  3.0&    0.275E-02 &  1.9 &     14 \\
 5&    0.327E-04 &  3.0&    0.709E-03 &  2.0 &     22 \\
 6&    0.410E-05 &  3.0&    0.182E-03 &  2.0 &     25 \\
\hline
\end{tabular}%
\end{table}%

In Table \ref{t02}, we employ the $\{P_2,P_1\}$ WG finite element \eqref{wk} and $\beta=8$.
   We can see from Table \ref{t02} that  we have an order three convergence in the $L^2$-norm and
  an order two convergence in the energy norm. In addition, the number of iterations needed for the 16-subdomain iteration is
  slightly higher than that of the 4-subdomain iteration.

\begin{table}[ht]
  \centering   \renewcommand{\arraystretch}{1.05}
  \caption{Error profile of $\{P_3,P_2\}$ solutions for \eqref{s-1}  on
      Figures \ref{grid1}--\ref{grid2} grids. }
\label{t03}
\begin{tabular}{c|cc|cc|r}
\hline
Grid & $\|Q_0 u - u_0 \|_0 $ &rate & $\|\nabla_w (Q_h u - u_h)\|_0$ &rate &\# iteration   \\
\hline
 &\multicolumn{5}{c}{The $\{P_3,P_2\}$ WG element \eqref{wk}, by 4-subdomain iteration.} \\
 \hline
 1&    0.366E-01 &  0.0&    0.279E-01 &  0.0 &     10 \\
 2&    0.236E-02 &  4.0&    0.632E-02 &  2.1 &     12 \\
 3&    0.151E-03 &  4.0&    0.947E-03 &  2.7 &     16 \\
 4&    0.947E-05 &  4.0&    0.126E-03 &  2.9 &     25 \\
 5&    0.592E-06 &  4.0&    0.162E-04 &  3.0 &     44 \\
\hline
 &\multicolumn{5}{c}{The $\{P_3,P_2\}$ WG element \eqref{wk}, by 16-subdomain iteration.} \\
 \hline
 1&    0.366E-01 &  0.0&    0.279E-01 &  0.0 &     10 \\
 2&    0.236E-02 &  4.0&    0.631E-02 &  2.1 &     14 \\
 3&    0.151E-03 &  4.0&    0.939E-03 &  2.7 &     19 \\
 4&    0.946E-05 &  4.0&    0.125E-03 &  2.9 &     35 \\
 5&    0.592E-06 &  4.0&    0.161E-04 &  3.0 &     59 \\
\hline
\end{tabular}%
\end{table}%

In Table \ref{t03}, $\beta=8$ and  $\{P_3,P_2\}$ weak Galerkin finite element are taken
for the 4-subdomain  and 16-subdomain iterations respectively. It can be seen from Table \ref{t03} that an order four convergence in the $L^2$-norm and an order three convergence in the energy norm are observed.  Table \ref{t03} shows that the number of iterations for the 16-subdomain iteration is somewhat higher than that of the 4-subdomain iteration. Note that the two iterations are the same as there are only 4 squares on the first level grid.

\begin{table}[ht]
  \centering   \renewcommand{\arraystretch}{1.05}
  \caption{Error profile of $\{P_4,P_3\}$ solutions  for \eqref{s-1}  on
      Figures \ref{grid1}--\ref{grid2} grids. }
\label{t04}
\begin{tabular}{c|cc|cc|r}
\hline
Grid & $\|Q_0 u - u_0 \|_0 $ &rate & $\|\nabla_w (Q_h u - u_h)\|_0$ &rate &\# iteration   \\
\hline
 &\multicolumn{5}{c}{By $\{P_4,P_3\}$ element \eqref{wk}, 4-subdomain iteration, $\beta=32$.} \\
 \hline
 1&    0.579E-02 &  0.0&    0.123E-01 &  0.0 &     29 \\
 2&    0.243E-03 &  4.6&    0.114E-02 &  3.4 &     55 \\
 3&    0.861E-05 &  4.8&    0.802E-04 &  3.8 &     96 \\
 4&    0.278E-06 &  5.0&    0.522E-05 &  3.9 &    121 \\
 5&    0.889E-08 &  5.0&    0.331E-06 &  4.0 &     78 \\
\hline
 &\multicolumn{5}{c}{By $\{P_4,P_3\}$ element \eqref{wk}, 16-subdomain iteration, $\beta=32$.} \\
 \hline
 1&    0.579E-02 &  0.0&    0.123E-01 &  0.0 &     29 \\
 2&    0.244E-03 &  4.6&    0.115E-02 &  3.4 &     59 \\
 3&    0.861E-05 &  4.8&    0.802E-04 &  3.8 &     97 \\
 4&    0.278E-06 &  5.0&    0.522E-05 &  3.9 &    121 \\
 5&    0.877E-08 &  5.0&    0.331E-06 &  4.0 &    139 \\
\hline
\end{tabular}%
\end{table}%

 The numerical results  for
 the $\{P_4,P_3\}$, $\{P_5,P_4\}$ and $\{P_6,P_5\}$ WG finite element solutions
 by the 4-subdomain  and 16-subdomain iterations are respectively listed in Tables \ref{t04}-\ref{t06} with corresponding
 $\beta=32$, $\beta=19$ and $\beta=32$.
In all these computations, we have observed the optimal order of convergence in the $L^2$-norm  and the energy norm.
It seems that larger $\beta$ may reduce the number of iterations for higher order finite elements.
It is surprising that $78$-iteration shows up for the $\{P_4,P_3\}$-element
   with 4-subdomain iterations on the fifth-level grid.
Another noticeable surprise is that the number of iterations for the 4-subdomain iteration
    increases much less than that of the 16-subdomain iteration, when the $\{P_5,P_4\}$ and $\{P_6,P_5\}$
    finite elements move to higher level grids. We conjecture that it might be due to non-hanging subdomains in the 4-subdomain iterations.

\begin{table}[ht]
  \centering   \renewcommand{\arraystretch}{1.05}
  \caption{Error profile of $\{P_5,P_4\}$ solutions  for \eqref{s-1}  on
      Figures \ref{grid1}--\ref{grid2} grids. }
\label{t05}
\begin{tabular}{c|cc|cc|r}
\hline
Grid & $\|Q_0 u - u_0 \|_0 $ &rate & $\|\nabla_w (Q_h u - u_h)\|_0$ &rate &\# iteration   \\
\hline
 &\multicolumn{5}{c}{By $\{P_5,P_4\}$ element \eqref{wk}, 4-subdomain iteration, $\beta=19$.} \\
 \hline
 1&    0.146E-02 &  0.0&    0.394E-02 &  0.0 &     31 \\
 2&    0.357E-04 &  5.4&    0.154E-03 &  4.7 &     55 \\
 3&    0.621E-06 &  5.8&    0.513E-05 &  4.9 &     71 \\
 4&    0.996E-08 &  6.0&    0.165E-06 &  5.0 &     78 \\
\hline
 &\multicolumn{5}{c}{By $\{P_5,P_4\}$ element \eqref{wk}, 16-subdomain iteration, $\beta=19$.} \\
 \hline
 1&    0.146E-02 &  0.0&    0.394E-02 &  0.0 &     31 \\
 2&    0.357E-04 &  5.4&    0.154E-03 &  4.7 &     59 \\
 3&    0.621E-06 &  5.8&    0.513E-05 &  4.9 &     81 \\
 4&    0.996E-08 &  6.0&    0.165E-06 &  5.0 &    104 \\
\hline
\end{tabular}%
\end{table}%

\begin{table}[ht]
  \centering   \renewcommand{\arraystretch}{1.05}
  \caption{Error profile of $\{P_6,P_5\}$ solutions  for \eqref{s-1}  on
      Figures \ref{grid1}--\ref{grid2} grids. }
\label{t06}
\begin{tabular}{c|cc|cc|r}
\hline
Grid & $\|Q_0 u - u_0 \|_0 $ &rate & $\|\nabla_w (Q_h u - u_h)\|_0$ &rate &\# iteration   \\
\hline
 &\multicolumn{5}{c}{By $\{P_6,P_5\}$ element \eqref{wk}, 4-subdomain iteration, $\beta=32$.} \\
 \hline
 1&    0.342E-03 &  0.0&    0.711E-03 &  0.0 &     43 \\
 2&    0.348E-05 &  6.6&    0.129E-04 &  5.8 &     67 \\
 3&    0.289E-07 &  6.9&    0.211E-06 &  5.9 &     99 \\
 4&    0.232E-09 &  7.0&    0.337E-08 &  6.0 &     94 \\
\hline
 &\multicolumn{5}{c}{By $\{P_6,P_5\}$ element \eqref{wk}, 16-subdomain iteration, $\beta=32$.} \\
 \hline
 1&    0.342E-03 &  0.0&    0.711E-03 &  0.0 &     43 \\
 2&    0.348E-05 &  6.6&    0.129E-04 &  5.8 &     97 \\
 3&    0.289E-07 &  6.9&    0.211E-06 &  5.9 &    133 \\
 4&    0.232E-09 &  7.0&    0.334E-08 &  6.0 &    156 \\
\hline
\end{tabular}%
\end{table}%

\subsection{Test Example 2}  We solve the elliptic boundary value model problem \eqref{model} where the configuration is set up as follows:  $a(x,y)=1$; $c=0$; the exact solution  \an{\label{s-2} u(x,y)=4  (x-x^3)  (y-y^3);  }
 and  the domain   $\Omega=(0,1)^2$.
The polygonal  grids of quadrilaterals and pentagons,  shown as in Figures \ref{grid3} and \ref{grid4},  are employed in this test. The 4-subdomain iterations and 16-subdomain iterations are computed respectively.

\begin{figure}[ht] \begin{center}
\begin{picture}(320,110)(0,0)
\put(0,-168){\includegraphics[width=5in]{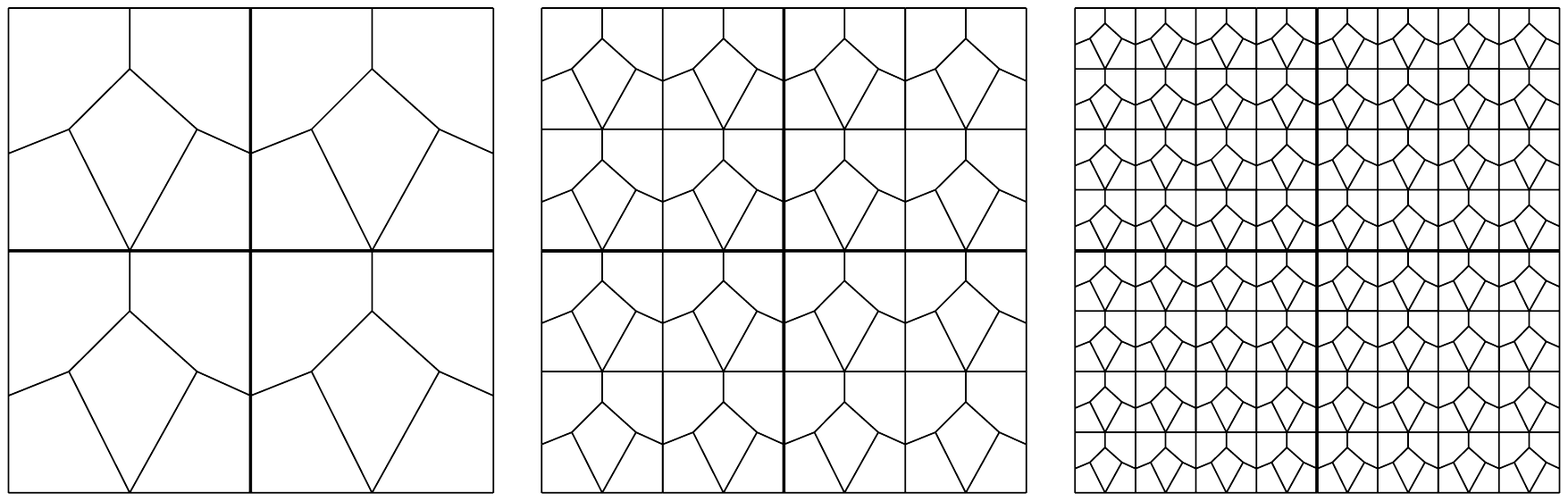}}  \end{picture}
\caption{ The first three levels of grids for the 4-subdomain iteration
    in Tables \ref{t21}--\ref{t22}.  }
\label{grid3} \end{center}
\end{figure}

\begin{figure}[ht] \begin{center}
\begin{picture}(320,110)(0,0)
\put(0,-168){\includegraphics[width=5in]{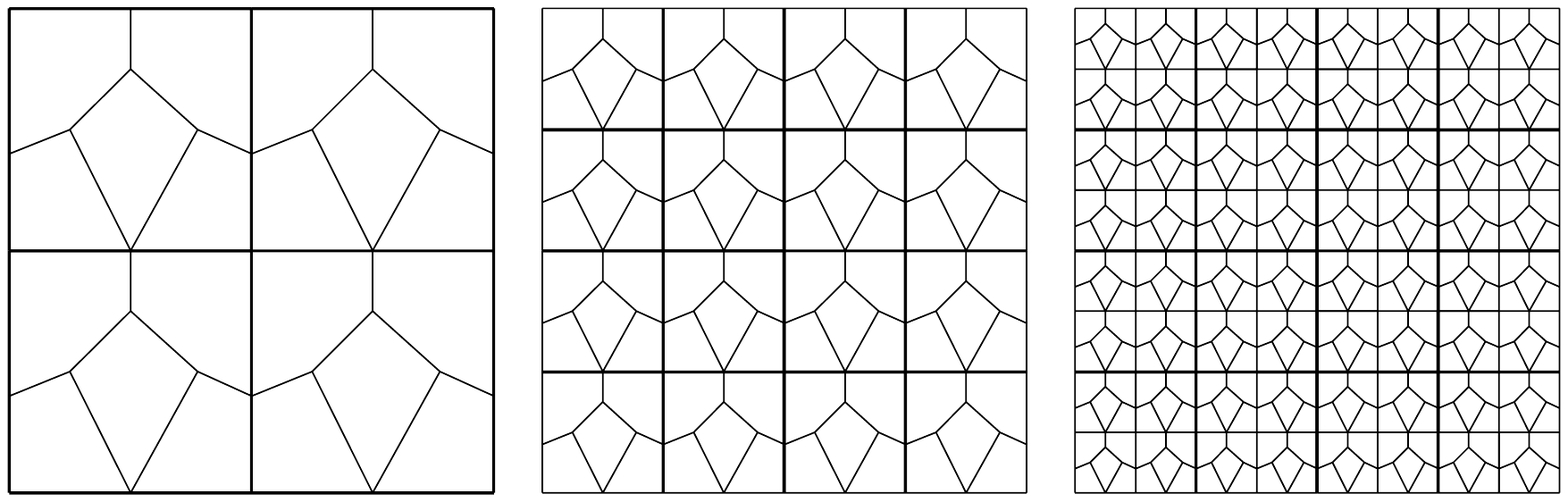}}  \end{picture}
\caption{ The first three levels of grids for the 16-subdomain iteration
    in Tables \ref{t21}--\ref{t22}.  }
\label{grid4} \end{center}
\end{figure}

The computational  results for the $\{P_2,P_{1}\}$ weak Galerkin finite element are listed in Table \ref{t21}.
Note that $Q_h u=\{ Q_0 u, Q_b u \}$  where $Q_0$ is the element-wise $L^2$ projection to the space $P_k(T)$ and $Q_b$ is the edge-wise $L^2$ projection to the space $P_{k-1}(e)$.
In the computation,  the iterative process is stopped when the iterative error achieves the truncation error.
The numerical solution converges at the optimal order in the $L^2$-norm  and in the energy norm respectively.  The number of iterations for the 4-subdomain method is slightly less than that of the 16-subdomain method.

\begin{table}[ht]
  \centering   \renewcommand{\arraystretch}{1.05}
  \caption{Error profile of $\{P_2,P_1\}$ \eqref{wk} solutions   on
      Figures \ref{grid3}--\ref{grid4} grids, and the number of iterations with
   4-subdomains and 16-subdomains. }
\label{t21}
\begin{tabular}{c|cc|cc|rr}
\hline
Grid & $\|Q_0 u - u_0 \|_0 $ &rate & $\|\nabla_w (Q_h u - u_h)\|_0$ &rate
&\multicolumn{2}{c}{\# iteration }  \\
\hline
 &\multicolumn{4}{c|}{The $\{P_2,P_1\}$ WG element } & 4-sub & 16-sub  \\
 \hline
 1&    0.180E-01 &  0.0&    0.817E-01 &  0.0 &     20 & 20\\
 2&    0.228E-02 &  3.0&    0.202E-01 &  2.0 &     24 & 41\\
 3&    0.287E-03 &  3.0&    0.491E-02 &  2.0 &     32 &  54\\
 4&    0.349E-04 &  3.0&    0.121E-02 &  2.0 &     43 &  68\\
 5&    0.432E-05 &  3.0&    0.299E-03 &  2.0 &     54  &   84\\
 6&    0.543E-06 &  3.0&    0.745E-04 &  2.0 &     60 &     99 \\
\hline
\end{tabular}%
\end{table}%

 The errors for the $\{P_3,P_{2}\}$ and $\{P_4,P_{3}\}$ weak Galerkin finite elements \eqref{wk} are listed in Table \ref{t22}.
The numerical solutions converge at the optimal order in the $L^2$-norm  and in the energy norm respectively. Due to the use of polygonal meshes,  the round-off error is accumulated to be very large
 for the computation of $\{P_4,P_{3}\}$ element on the 4th-level grid.
We observe that for the high-order finite elements,  the number of iteration may not
 increase when the  number of the grid level  increases.

\begin{table}[ht]
  \centering   \renewcommand{\arraystretch}{1.05}
  \caption{Error profile of $\{P_k,P_{k-1}\}$ \eqref{wk} solutions  on
      Figures \ref{grid3}--\ref{grid4} grids, and the number of iterations with
   4-subdomains and 16-subdomains. }
\label{t22}
\begin{tabular}{c|cc|cc|rr}
\hline
Grid & $\|Q_0 u - u_0 \|_0 $ &rate & $\|\nabla_w (Q_h u - u_h)\|_0$ &rate
&\multicolumn{2}{c}{\# iteration }  \\
\hline
 &\multicolumn{4}{c|}{The $\{P_3,P_2\}$ WG element } & 4-sub & 16-sub  \\
 \hline
 1&    0.174E-02 &  0.0&    0.933E-02 &  0.0 &     39 &     39 \\
 2&    0.110E-03 &  4.0&    0.126E-02 &  2.9 &     51 &     55 \\
 3&    0.687E-05 &  4.0&    0.162E-03 &  3.0 &     84 &     89 \\
 4&    0.438E-06 &  4.0&    0.207E-04 &  3.0 &     59 &     86 \\
 5&    0.278E-07 &  4.0&    0.260E-05 &  3.0 &     69 &     115\\
\hline
 &\multicolumn{4}{c|}{The $\{P_4,P_3\}$ WG element } & 4-sub & 16-sub  \\
 \hline
 1&    0.115E-03 &  0.0&    0.417E-03 &  0.0 &     44 &     44 \\
 2&    0.364E-05 &  5.0&    0.274E-04 &  3.9 &     64 &     64 \\
 3&    0.114E-06 &  5.0&    0.176E-05 &  4.0 &     99 &     93 \\
 4&    0.469E-08 &  ---&    0.113E-06 &  4.0 &     95 &    111 \\
\hline
\end{tabular}%
\end{table}%

\subsection{Test Example 3} We shall test the subdomain iterations for a superconvergent WG finite element method \cite{zhang}. We consider the Poisson equation  on the domain $\Omega=(0,1)^2$ with the exact solution
$u(x,y)=\sin(\pi x)\sin(\pi y)$. The superconvergent $\{P_k, P_k\}$ WG finite element was introduced in \cite{zhang} as follows:
  \begin{equation}\label{F2}
      \{ \{v_0, v_b\} : v_0 \in P_k(T), \ v_b\in P_{k}(e), \ e\subset T, \
                 T\in \T_h \}.
  \end{equation}
For this superconvergent $\{P_k, P_k\}$ WG finite element, the weak gradient is defined by
   $\nabla_w v_h\in \b{RT}_k= P_{k}(T)^2 + \b x P_{k}(T)$,
   satisfying
\an{\label{G2} ( \nabla_w v_h, \b q)_T = -(v_0, \t{div} \b q)_T
   +\langle v_b, \b q\cdot \b n_T\rangle_{\partial T},\quad\forall \b q\in \b{RT}_k.  }
For this WG finite element,  the stabilizer $s(\cdot, \cdot)$ is dropped from the WG scheme \eqref{wg}, in order to get one-order superconvergence as discussed in \cite{zhang}.

We take  $\beta=4$ in all  $\{P_k,P_k\}$ weak Galerkin finite element computations.
Note that $Q_h u=\{ Q_0 u, Q_b u\}$ where  $Q_0 u$ is the triangle-wise $L^2$ projection of $u$ onto the polynomial space $P_k(T)$ and $Q_b$ is the edge-wise  $L^2$ projection of $u$ onto the polynomial space $P_k(e)$.
 Again the iteration is stopped when the iterative error is about the truncation error.
In Tables \ref{t1}-\ref{t4}, we have observed one-order superconvergence for the $\{P_k,P_k\}$ weak Galerkin finite element
defined in  \eqref{F2}, in both the $L^2$-norm and the energy norm.

As expected, the number of 16-subdomain iterations would be more than that of
4-subdomain iterations for  the $\{P_1,P_1\}$ weak Galerkin finite element as we can see from Table \ref{t1}.

\begin{table}[ht]
  \centering   \renewcommand{\arraystretch}{1.05}
  \caption{Error profile for the $\{P_1,P_1\}$ solution   on
      Figures \ref{grid1}--\ref{grid2} grids. }
\label{t1}
\begin{tabular}{c|cc|cc|r}
\hline
Grid & $\|Q_0 u - u_0 \|_0 $ &rate & $\|\nabla_w (Q_h u - u_h)\|_0$ &rate &\# iteration   \\
\hline
 &\multicolumn{5}{c}{The $\{P_1,P_1\}$ WG element \eqref{F2}, by 4-subdomain iteration.} \\
 \hline
 1&    0.208E-01 &  0.0&    0.500E+00 &  0.0 &      7 \\
 2&    0.305E-02 &  2.8&    0.140E+00 &  1.8 &      9 \\
 3&    0.393E-03 &  3.0&    0.362E-01 &  2.0 &     13 \\
 4&    0.494E-04 &  3.0&    0.912E-02 &  2.0 &     12 \\
 5&    0.624E-05 &  3.0&    0.229E-02 &  2.0 &     17 \\
 6&    0.789E-06 &  3.0&    0.572E-03 &  2.0 &     26 \\
 7&    0.994E-07 &  3.0&    0.143E-03 &  2.0 &     42 \\
\hline
 &\multicolumn{5}{c}{The $\{P_1,P_1\}$ WG element \eqref{F2}, by 16-subdomain iteration.} \\
 \hline
 1&    0.208E-01 &  0.0&    0.500E+00 &  0.0 &      9 \\
 2&    0.304E-02 &  2.8&    0.140E+00 &  1.8 &     11 \\
 3&    0.392E-03 &  3.0&    0.362E-01 &  2.0 &     13 \\
 4&    0.496E-04 &  3.0&    0.913E-02 &  2.0 &     13 \\
 5&    0.631E-05 &  3.0&    0.229E-02 &  2.0 &     18 \\
 6&    0.795E-06 &  3.0&    0.572E-03 &  2.0 &     30 \\
 7&    0.101E-06 &  3.0&    0.143E-03 &  2.0 &     48 \\
\hline
\end{tabular}%
\end{table}%

In Table \ref{t2}, we list the computational  results for the $\{P_2,P_2\}$ weak Galerkin finite element.
Worse than the $\{P_1,P_1\}$ computation, the number of iterations for 16-subdomains is much more than that of
4-subdomain. But on the other side,  a better parallelization is possible for the computation with 16-subdomains.
 
\begin{table}[ht]
  \centering   \renewcommand{\arraystretch}{1.05}
  \caption{Error profile for the $\{P_2,P_2\}$ solution   on
      Figures \ref{grid1}--\ref{grid2} grids. }\label{t2}
\begin{tabular}{c|cc|cc|r}
\hline
Grid & $\|Q_0 u - u_0 \|_0 $ &rate & $\|\nabla_w (Q_h u - u_h)\|_0$ &rate &\# iteration   \\
\hline
 &\multicolumn{5}{c}{The $\{P_2,P_2\}$ WG element \eqref{F2}, by 4-subdomain iteration.} \\
 \hline
 1&    0.320E-02 &  0.0&    0.116E+00 &  0.0 &     12 \\
 2&    0.203E-03 &  4.0&    0.158E-01 &  2.9 &     12 \\
 3&    0.127E-04 &  4.0&    0.202E-02 &  3.0 &     14 \\
 4&    0.802E-06 &  4.0&    0.254E-03 &  3.0 &     25 \\
 5&    0.507E-07 &  4.0&    0.318E-04 &  3.0 &     48 \\
 6&    0.321E-08 &  4.0&    0.398E-05 &  3.0 &     92 \\
\hline
 &\multicolumn{5}{c}{The $\{P_2,P_2\}$ WG element \eqref{F2}, by 16-subdomain iteration.} \\
 \hline
 1&    0.320E-02 &  0.0&    0.116E+00 &  0.0 &     16 \\
 2&    0.202E-03 &  4.0&    0.158E-01 &  2.9 &     12 \\
 3&    0.128E-04 &  4.0&    0.202E-02 &  3.0 &     15 \\
 4&    0.801E-06 &  4.0&    0.254E-03 &  3.0 &     29 \\
 5&    0.508E-07 &  4.0&    0.318E-04 &  3.0 &     55 \\
 6&    0.321E-08 &  4.0&    0.398E-05 &  3.0 &    107 \\
\hline
\end{tabular}%
\end{table}%

Table \ref{t3} shows that the number of iterations for 16-subdomains is about the same as that of
4-subdomains for the $\{P_3,P_3\}$ weak Galerkin finite element.

\begin{table}[ht]
  \centering   \renewcommand{\arraystretch}{1.05}
  \caption{Error profile for the $\{P_3,P_3\}$ solution on
      Figures \ref{grid1}--\ref{grid2} grids. }\label{t3}
\begin{tabular}{c|cc|cc|r}
\hline
Grid & $\|Q_0 u - u_0 \|_0 $ &rate & $\|\nabla_w (Q_h u - u_h)\|_0$ &rate &\# iteration   \\
\hline
 &\multicolumn{5}{c}{The $\{P_3,P_3\}$ WG element \eqref{F2}, by 4-subdomain iteration.} \\
 \hline
 1&    0.478E-03 &  0.0&    0.222E-01 &  0.0 &     12 \\
 2&    0.179E-04 &  4.7&    0.149E-02 &  3.9 &     24 \\
 3&    0.593E-06 &  4.9&    0.944E-04 &  4.0 &     44 \\
 4&    0.189E-07 &  5.0&    0.593E-05 &  4.0 &     82 \\
 5&    0.597E-09 &  5.0&    0.372E-06 &  4.0 &    142 \\
\hline
 &\multicolumn{5}{c}{The $\{P_3,P_3\}$ WG element \eqref{F2}, by 16-subdomain iteration.} \\
 \hline
 1&    0.478E-03 &  0.0&    0.222E-01 &  0.0 &     12 \\
 2&    0.179E-04 &  4.7&    0.149E-02 &  3.9 &     22 \\
 3&    0.593E-06 &  4.9&    0.944E-04 &  4.0 &     46 \\
 4&    0.189E-07 &  5.0&    0.593E-05 &  4.0 &     80 \\
 5&    0.596E-09 &  5.0&    0.372E-06 &  4.0 &    161 \\
\hline
\end{tabular}%
\end{table}%

Finally   the computational results by using $\{P_4,P_4\}$, $\{P_5,P_5\}$, $\{P_6,P_6\}$ weak Galerkin finite elements are illustrated in Table \ref{t4}. In all these cases, the number of 16-subdomain iterations is slightly bigger than that of the 4-subdomain iterations.

\begin{table}[ht]
  \centering   \renewcommand{\arraystretch}{1.05}
  \caption{Error profiles for the $P_4$,$P_5$ and $P_6$ WG solution  on
      Figures \ref{grid1}--\ref{grid2} grids. }\label{t4}
\begin{tabular}{c|cc|cc|r}
\hline
Grid & $\|Q_0 u - u_0 \|_0 $ &rate & $\|\nabla_w (Q_h u - u_h)\|_0$ &rate &\# iteration   \\
\hline
 &\multicolumn{5}{c}{The $\{P_4,P_4\}$ WG element \eqref{F2}, by 4-subdomain iteration.} \\
 \hline
 2&    0.120E-05 &  6.0&    0.122E-03 &  4.9 &     34 \\
 3&    0.190E-07 &  6.0&    0.388E-05 &  5.0 &     54 \\
 4&    0.301E-09 &  6.0&    0.122E-06 &  5.0 &    123 \\
\hline
 &\multicolumn{5}{c}{The $\{P_4,P_4\}$ WG element \eqref{F2}, by 16-subdomain iteration.} \\
 \hline
 2&    0.119E-05 &  6.0&    0.122E-03 &  4.9 &     24 \\
 3&    0.190E-07 &  6.0&    0.387E-05 &  5.0 &     62 \\
 4&    0.300E-09 &  6.0&    0.123E-06 &  5.0 &    133 \\
\hline
 &\multicolumn{5}{c}{The $\{P_5,P_5\}$ WG element \eqref{F2}, by 4-subdomain iteration.} \\
 \hline
 1&    0.938E-05 &  0.0&    0.541E-03 &  0.0 &     18 \\
 2&    0.827E-07 &  6.8&    0.896E-05 &  5.9 &     36 \\
 3&    0.655E-09 &  7.0&    0.141E-06 &  6.0 &     95 \\
\hline
 &\multicolumn{5}{c}{The $\{P_5,P_5\}$ WG element \eqref{F2}, by 16-subdomain iteration.} \\
 \hline
 1&    0.938E-05 &  0.0&    0.541E-03 &  0.0 &     18 \\
 2&    0.816E-07 &  6.8&    0.890E-05 &  5.9 &     44 \\
 3&    0.647E-09 &  7.0&    0.141E-06 &  6.0 &    111 \\
\hline
 &\multicolumn{5}{c}{The $\{P_6,P_6\}$ WG element \eqref{F2}, by 4-subdomain iteration.} \\
 \hline
 1&    0.115E-05 &  0.0&    0.710E-04 &  0.0 &     24 \\
 2&    0.460E-08 &  8.0&    0.579E-06 &  6.9 &     66 \\
 3&    0.185E-10 &  8.0&    0.459E-08 &  7.0 &    162 \\
\hline
 &\multicolumn{5}{c}{The $\{P_6,P_6\}$ WG element \eqref{F2}, by 16-subdomain iteration.} \\
 \hline
 1&    0.115E-05 &  0.0&    0.710E-04 &  0.0 &     24 \\
 2&    0.459E-08 &  8.0&    0.578E-06 &  6.9 &     72 \\
 3&    0.185E-10 &  8.0&    0.461E-08 &  7.0 &    166 \\
\hline
\end{tabular}%
\end{table}%

\newpage



\end{document}